\documentclass[12pt]{article}
\usepackage{bbm}
\usepackage{latexsym}
\usepackage{mathrsfs}
\usepackage{graphicx}
\usepackage{subfigure}
\usepackage{amssymb}
\usepackage{amsmath}
\usepackage{amsthm}
\usepackage{color}
\usepackage{cite}
\usepackage{float}
\usepackage{indentfirst}
\usepackage{anysize}\marginsize{45mm}{45mm}{40mm}{50mm}
\usepackage{caption}
\usepackage{float}

\usepackage{enumerate}

\newtheoremstyle{lemma}{\topsep}{\topsep}%
     {}
     {}
     {\bfseries}
     {}
     {0.1em}
     {\thmname{#1}\thmnumber{ #2}\thmnote{ #3}}
\theoremstyle{lemma}  

\newtheorem{theorem}{Theorem}[section]    
\newtheorem{lemma}[theorem]{Lemma}
\newtheorem{corollary}[theorem]{Corollary}

\numberwithin{equation}{section}

\usepackage{latexsym, bm}
\title
{Continuous anti-forcing spectra of cata-condensed hexagonal systems\thanks{Supported by NSFC (grant no. 11371180).} }
\author{Kai Deng$^{a,b}$, { Heping Zhang$^{a,}$\thanks{
Corresponding author. \newline
\emph{E-mail address}: zhanghp@lzu.edu.cn, dengkai04@126.com}}\\
{\footnotesize $^{a}$School of Mathematics and Statistics, Lanzhou University, Lanzhou, Gansu 730000, P. R.~China}\\
{\footnotesize $^{b}$School of Mathematics and Information Science, Beifang University of Nationalities,}\\{\footnotesize Yinchuan, Ningxia 750027, P. R.~China}}
\date{}
\begin{document}

\maketitle
\begin{abstract}
The anti-forcing number of a perfect matching $M$ of a graph $G$
is the minimal number of edges not in $M$ whose removal make $M$ as a unique perfect matching of the resulting graph.
The anti-forcing spectrum of $G$ is the set of
anti-forcing numbers of all perfect matchings of $G$.
In this paper we prove that the anti-forcing spectrum of any cata-condensed hexagonal system is continuous,
that is, it is an integer interval.

\vskip 0.1 in


\end{abstract}

\section{Introduction}
Let $G$ be a simple graph with vertex set $V(G)$ and edge set $E(G)$. A {\em perfect matching} or 1-factor
of $G$ is a set of disjoint edges which covers all vertices of $G$.
Harary et al. \cite{Harary} proposed the \emph{forcing number} of a perfect matching $M$ of a graph $G$.
The roots of this concept can be found in an earlier paper by Klein and Randi\'{c} \cite{Klein}.
There, the forcing number has been called the \emph{innate degree of freedom} of a  Kekul\'{e} structure.
The forcing number of a perfect matching $M$ of a graph $G$ is equal to the smallest cardinality of
some subset $S$ of $M$ such that $M$ is completely determined by this subset (i.e., $S$ is not contained in other perfect matchings of $G$).
The {\em minimum} (resp. {\em maximum}) \emph{forcing number} of $G$ is the minimum (resp. maximum) value over forcing numbers
of all perfect matchings of $G$.
The set of forcing numbers
of all perfect matchings of $G$ is called the \emph{forcing spectrum} of $G$ \cite{Adams}.
The sum of forcing numbers of all perfect matchings of $G$ is called
the \emph{degree of freedom} of $G$,
which is relative to Clar's resonance-theoretic ideals \cite{Clar}.
For more results on the matching forcing  problem,
we refer the reader to 
\cite{c60, c70, c72, Afshani, Jiang1, Jiang, Ye, ZhangD, Lam, Pachter, Wang, Xu, Matthew}.

In 2007, Vuki\v{c}evi\'{c} and Trinajsti\'{c} \cite{VT1} introduced the \emph{anti-forcing number}
that is opposite to the forcing number.
The anti-forcing number of a graph $G$ is the smallest number of edges
whose removal result in a subgraph of $G$ with a unique perfect matching.
After this initial report, several papers appeared on this topic \cite{VT2, Deng1, Deng2, ZhangQQ, Che}.

Recently, Lei, Yeh and Zhang  \cite{Zh2} define the \emph{anti-forcing number of a perfect matching} $M$ of a graph $G$
as the minimal number of edges not in $M$ whose removal to make $M$ as a single perfect matching of the resulting graph, denoted by $af(G,M)$.
By this definition,  the anti-forcing number of a graph $G$ is the smallest anti-forcing number over all perfect matchings
of $G$. Hence the anti-forcing number of $G$ is the \emph{minimum anti-forcing number} of $G$, denoted by $af(G)$.
Naturally, the \emph{maximum anti-forcing number} of $G$ is defined as the largest anti-forcing number over all perfect matchings
of $G$, denoted by $Af(G)$.
They \cite{Zh2} also  defined the \emph{anti-forcing spectrum} of $G$ as the set of anti-forcing numbers of all perfect matchings of $G$, and denoted by Spec$_{af}(G)$.
If Spec$_{af}(G)$ is an integer interval,
then the anti-forcing spectrum of $G$ is called to be \emph{continuous}.

Let $M$ be a perfect matching of a graph $G$.
A cycle $C$ of $G$ is called an \emph{$M$-alternating cycle} if  the edges of $C$ appear alternately in $M$ and $E(G)\backslash M$.
If $C$ is an $M$-alternating cycle of $G$, then the symmetric difference $M\triangle C:=(M-C)\cup(C-M)$ is another perfect matching of $G$.

A set $\mathcal{A}$ of $M$-alternating cycles of a graph $G$ is called a \emph{compatible $M$-alternating set} if any two members of $\mathcal{A}$
either are disjoint or intersect only at edges in $M$.
Let $c^{\prime}(M)$ denote the cardinality of a maximum compatible $M$-alternating set
of $G$. For a planar bipartite graph $G$ with a perfect matching $M$,
the following minimax theorem reveals the relationship between
$af(G,M)$ and $c^{\prime}(M)$.

\begin{theorem}\label{compatibleset}\cite{Zh2}{\bf .}
Let $G$ be a planar bipartite graph with a perfect matching $M$. Then $af(G,M)=c^{\prime}(M)$.
\end{theorem}

A \emph{hexagonal system} (or \emph{benzenoid system})\cite{Cyvin} is a finite 2-connected
planar bipartite graph in which each interior face is surrounded by a regular hexagon of side length one.
Hexagonal systems are of great important for theoretical chemistry since they are the molecular graphs
of benzenoid hydrocarbons.

Let $H$ be a hexagonal system with a perfect matching $M$.
A set of $M$-alternating hexagons (the intersection is allowed) of $H$ is called an \emph{$M$-alternating set}.
A \emph{Fries set} of $H$ is a maximum alternating set over all perfect matchings of $H$.
The size of a Fries set of $H$ is called the \emph{Fries number} of $H$ and denoted by $Fries(H)$ \cite{Fries}.
It is obvious that an $M$-alternating set of $H$ is also a compatible $M$-alternating set.
By Theorem \ref{compatibleset}, $Af(H)\geq Fries(H)$.
The following theorem implies that the equality holds.

\begin{theorem}\label{maxantiforcing}\cite{Zh2}{\bf .}
Let $H$ be a hexagonal system with a perfect matching.
Then $Af(H)=Fries(H)$.
\end{theorem}

In this paper we consider the anti-forcing spectra of
cata-condensed hexagonal systems.
In the next section,
we introduce some graph-theoretic terms relevant to our subject and
give some useful lemmas.
In Section 3,
we prove that the anti-forcing spectrum of any cata-condensed hexagonal system is continuous.
It is quite different from the case for forcing spectrum.
In fact,
the forcing spectra of some cata-condensed hexagonal systems
have gaps (see \cite{Zh2, ZhangD}).

\section{Preliminaries and lemmas}

The \emph{inner dual graph} $H^{*}$ of a hexagonal system $H$ is a graph whose
vertices correspond to hexagons of $H$, and two such vertices are adjacent by an edge of $H^*$ if and only if
they correspond to two adjacent hexagons (i.e., these two hexagons have one common edge).
Then $H$ is \emph{cata-condensed} if and only if $H^{*}$ is a tree \cite{Balaban}.

We can see that edges of a cata-condensed hexagonal system $H$ can be classified into
\emph{boundary} edges (edges are on the boundary of $H$) and \emph{shared} edges (edges are shared by two hexagons of $H$),
and all vertices of $H$ are on the boundary (i.e., $H$ has no inner vertices).
A subgraph $G^\prime$ of a graph $G$ is \emph{nice}
if $G-V(G^\prime)$ (the graph obtained by deleting vertices of $V(G^\prime)$ and their incident edges from $G$) has a perfect matching.
It is well known that every cata-condensed hexagonal system $H$ has perfect matchings and
every cycle in it is nice \cite{Cyvin}, so each hexagon of $H$ can be $M$-alternating with respect to some perfect matching $M$ of $H$.

\begin{figure*}[http]
\hfill\begin{minipage}[t]{0.4\textwidth}
\centering
\includegraphics[height=40mm]{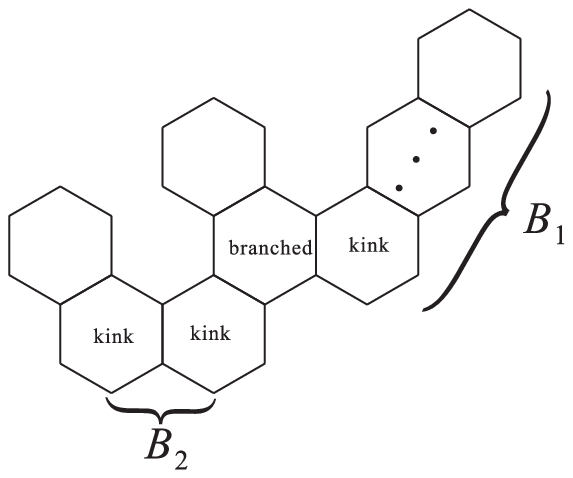}
\caption{A cata-condensed hexagonal system with one branched hexagon and
  three kinks.}
\label{kink}
\end{minipage}%
\hfill
  \begin{minipage}[t]{0.3\textwidth}
\centering
\includegraphics[height=10mm]{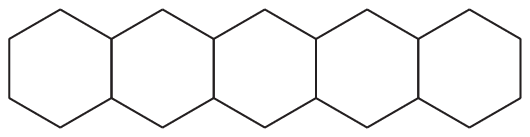}
\caption{The linear chain with five hexagons.}
\label{linear}
\end{minipage}%
\hspace*{\fill}
\end{figure*}

A hexagon $s$ of a cata-condensed hexagonal system $H$ has one, two, or three neighboring hexagons.
$s$ is called \emph{terminal} if it has one neighboring hexagon,
and \emph{branched} if it has three neighboring hexagons.
$s$ has exactly two neighboring hexagons is a \emph{kink} if $s$ possesses two adjacent vertices of degree 2,
is \emph{linear} otherwise.
An illustration is given in Fig. \ref{kink}.
A cata-condensed hexagonal system with no branched hexagons is called a \emph{hexagonal chain}.
A hexagonal chain with no kinks is called a \emph{linear chain},
an example is shown in Fig. \ref{linear}.

A linear chain $B$ contained in a cata-condensed hexagonal system $H$ is called \emph{maximal}
if $B$ is not contained in other linear chains of $H$.
For example, see Fig. \ref{kink}, $B_1$ and $B_2$ are two maximal linear chains.

Let $B$ be a maximal linear chain of a cata-condensed hexagonal system $H$.
We draw a straight line $l$ passing through the two centers of the two terminal hexagons of $B$.
Let $E$ be the set of those edges which intersecting $l$.
By the Lemma 2.1 in \cite{Zh3}, the following lemma is immediate.

\begin{lemma}\label{unique1}{\bf .}
Let $M$ be any perfect matching of $H$.
Then $|M\cap E|=1$.
\end{lemma}

Let $\mathcal{A}$ be a compatible $M$-alternating set with respect to a perfect matching $M$ of a planar bipartite graph $G$.
Two cycles $C_1$ and $C_2$ of $\mathcal{A}$ are \emph{crossing} if they share an edge $e$ in $M$ and the four edges adjacent to $e$
alternate in $C_1$ and $C_2$ (i.e., $C_1$ enters into $C_2$ from one side and leaves from the other side via $e$).
$\mathcal{A}$ \emph{is non-crossing}  if any two cycles in $\mathcal{A}$ are not crossing.
Lei, Yeh and Zhang  \cite{Zh2} proved that any compatible $M$-alternating set can be improved to be a non-crossing compatible $M$-alternating set with the same cardinality.
Let $H$ be a cata-condensed hexagonal system
with a perfect matching $M$.
For a cycle $C$ of $H$, let $h(C)$ denote the number of hexagons in the interior of $C$.
Then we can choose a maximum non-crossing compatible $M$-alternating set $\mathcal{A}$
such that $|\mathcal{A}|=af(H,M)$ and $h(\mathcal{A})=\sum_{C\in \mathcal{A}}h(C)$ ($h(\mathcal{A})$ is called \emph{$h$-index} of $\mathcal{A}$ \cite{Zh2}) is as small as possible.
By using those notations, we give the following lemma.

\begin{lemma}\label{compatible}{\bf .}
$\mathcal{A}$ contains all $M$-alternating hexagons of $H$, and
any two non-hexagon cycles of $\mathcal{A}$ have at most one common edge in $M$ and are interior disjoint (i.e., have no common areas).
\end{lemma}

\begin{proof}

Let $s$ be any $M$-alternating hexagon of $H$.
Suppose $s\notin \mathcal{A}$.
By the maximality of $|\mathcal{A}|$,
$\mathcal{A}\cup\{s\}$ is not a compatible $M$-alternating set.
So there is a cycle $C\in \mathcal{A}$ which is not compatible with $s$.
Since any two $M$-alternating hexagons must be compatible, $C$ is not a hexagon of $H$.
$s$ is in the interior of $C$ since $H$ is cata-condensed.
We claim that $(\mathcal{A}\setminus \{C\})\cup\{s\}$ is a compatible $M$-alternating set. Otherwise there is a cycle $C^\prime\in \mathcal{A}\setminus\{C\}$ such that $C^\prime$ and $s$ are not compatible.
Hence $s$ is in the interior of $C^\prime$.
It implies that $C$ and $C^\prime$ are either not compatible
or crossing, a contradiction.
Therefore, $(\mathcal{A}\setminus \{C\})\cup\{s\}$ is a maximum non-crossing compatible $M$-alternating set with smaller $h$-index,
a contradiction. Hence $s\in \mathcal{A}$.

Let $C_1$ and $C_2$  be two non-hexagon cycles in $\mathcal{A}$.
First we prove that $C_1$ and $C_2$ are interior disjoint.
If not, without loss of generality,
we may suppose $C_1$ is contained in the interior of $C_2$
since $C_1$ and $C_2$ are not crossing.
Then there is an $M$-alternating hexagon $s$ in the interior of $C_1$ since $C_1$ is $M$-alternating \cite{Guo, Zhang}.
By the above discussion,
we have that $s\in \mathcal{A}$ and $s$ is compatible with both  $C_1$ and $C_2$.
Since $H$ is cata-condensed,
$C_1$ must pass through the vertices of $s$ and the three edges of $s$
in $M$, but not the other three edges of $s$.
Hence $C_1$ must pass through the six edges going out of $s$.
Similarly, $C_2$ must pass through the six edges going out of $s$.
It implies that $C_1$ and $C_2$ are not compatible, a contradiction.

Next, we prove that $C_1$ and $C_2$
have at most one common edge in $M$.
If $C_1$ and $C_2$ have at least two common edges in $M$,
then there must generate inner vertices in $H$,
a contradiction.
\end{proof}

\begin{lemma}\label{big}{\bf .}
Let $H$ be a cata-condensed hexagonal system with at least two hexagons.
Then $af(H)<Fries(H)$.
\end{lemma}
\begin{proof}
By Theorem \ref{maxantiforcing}, $Af(H)=Fries(H)$.
It is sufficient to prove that $af(H)<Af(H)$.
Let $s$ be a terminal hexagon of $H$.
Then $s$ has a neighboring hexagon $s^\prime$
since $H$ has at least two hexagons.
Let $e$ be the sheared edge of $s$ and $s^\prime$.
Since $s^\prime$ is nice, there is a perfect matching $M$ of $H$ such that $s^\prime$ is $M$-alternating
and $e\in M$.
Note that $s$ is also $M$-alternating.
So $F=M\triangle s$ is a perfect matching of $H$ such that $s$ is $F$-alternating.
Let $\mathcal{A}$ be a maximum non-crossing compatible $F$-alternating set with smallest $h$-index.
By Lemma \ref{compatible}, we have that $s \in \mathcal{A}$.
We can see that  no cycle of $\mathcal{A}$ passing through
the three edges of $s^\prime$ not in $M$.
So $\mathcal{A}\cup \{s^\prime\}$ is a compatible $M$-alternating set.
By Theorem \ref{compatibleset},
we have that $af(H)\leq |\mathcal{A}|<|\mathcal{A}\cup \{s^\prime\}|\leq af(H,M)\leq Af(H)$.
\end{proof}

\section{Continuous anti-forcing spectra}

Let $a$ and $b$ be two integer numbers, and $a\leq b$.
In the following,
we use $[a,b]$ to denote the integer interval from $a$ to $b$.

\begin{figure}[http]
  \centering
   \includegraphics[height=50mm]{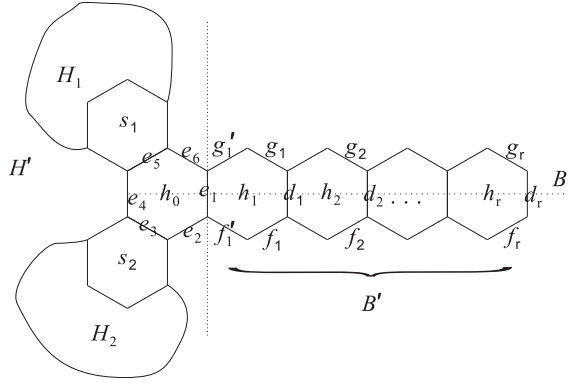}
      \caption{The illustration of the proof of Theorem \ref{continuous}.}
  \label{unique}
\end{figure}

\begin{theorem}\label{continuous}{\bf .}
Let $H$ be a cata-condensed hexagonal system. Then anti-forcing spectrum of $H$ is continuous.
\end{theorem}

\begin{proof}
We proceed by induction on the number $n$ of hexagons of $H$.
If $H$ is a single hexagon, then Spec$_{af}(H)=\{1\}$.
Suppose $n\geq 2$.
Take a maximal linear chain $B$ in $H$ such that one end hexagon of $B$ is a terminal hexagon of $H$.
Let $h_0,h_1,\ldots,h_r$ ($r\geq1$) be hexagons of $B$ in turn, and $h_r$ be terminal (see Fig. \ref{unique}).

If $h_0$ is also a terminal hexagon of $H$,
then $H=B$ is a linear chain with $n>1$ hexagons.
We can  check that Spec$_{af}(H)=[1,2]$.

If $h_0$ is not terminal,
then $h_0$ is a kink or branched hexagon of $H$.
Let $B^{\prime}$ be the linear chain obtained from $B$
by removing hexagon $h_0$
and $H^\prime$  the cata-condensed hexagonal system
obtained from $H$ by removing the hexagons of $B^\prime$ (see Fig. \ref{unique}).
Then $H^{\prime}$ has less than $n$ hexagons.
By the induction hypothesis,
the anti-forcing spectrum of $H^{\prime}$ is continuous.
By Theorem \ref{maxantiforcing}, $Af(H^{\prime})=Fries(H^{\prime})$.
Let $af(H^{\prime})=a^{\prime}$.
Then Spec$_{af}(H^{\prime})=[a^{\prime},Fries(H^{\prime})]$.

\noindent\textbf{Claim 1.} $[a^{\prime}+1,Fries(H^{\prime})]\subseteq$Spec$_{af}(H)$.

Since $h_0$ is not terminal, $H^\prime$ has at least two hexagons.
By Lemma \ref{big}, $a^{\prime}+1\leq Fries(H^{\prime})$.
For any $i\in [a^{\prime}+1,Fries(H^{\prime})]$, we want to prove $i\in$Spec$_{af}(H)$.
Since  $i-1\in [a^{\prime},Fries(H^{\prime})-1]$,
by the induction hypothesis,
there is a perfect matching $M^{\prime}$
of $H^{\prime}$ such that $af(H^{\prime},M^{\prime})=|\mathcal{A}^{\prime}|=i-1\geq1$,
where $\mathcal{A}^{\prime}$ is a maximum non-crossing compatible $M^{\prime}$-alternating set of $H^{\prime}$ with smallest $h$-index.
Note that $M=$$M^{\prime}$$\cup $$\{f_1,$$f_2,$$\ldots,$$f_r,$$g_1,$$g_2,$$\ldots,$
$g_r\}$
is a perfect matching of $H$ and $d_i\notin M$ for each
$1\leq i\leq r$.
By Lemma \ref{unique1}, either $e_4\in M^{\prime}$ or $e_1\in M^{\prime}$.

If $e_1\in M^{\prime}$, then $h_1$ is $M$-alternating.
By Lemma \ref{compatible}, $h_1\in \mathcal{A}$,
where $\mathcal{A}$ is a maximum non-crossing compatible $M$-alternating set of $H$ with smallest $h$-index.
We can see that $\mathcal{A}\setminus\{h_1\}$ is a compatible $M^{\prime}$-alternating set of $H^{\prime}$,
and $|\mathcal{A}|=|\mathcal{A}\setminus\{h_1\}|+1\leq |\mathcal{A}^{\prime}|+1=i$.
On the other hand, $\mathcal{A}^{\prime}\cup\{h_1\}$ is a compatible $M$-alternating set of $H$,
so $i=|\mathcal{A}^{\prime}|+1=|\mathcal{A}^{\prime}\cup\{h_1\}|\leq|\mathcal{A}|$,
i.e., $|\mathcal{A}|=i$.
By Theorem \ref{compatibleset}, $af(H,M)=i\in$Spec$_{af}(H)$.

From now on, we suppose $e_4\in M^{\prime}$.
Then $\{e_2,e_6\}\subseteq M^{\prime}$.
So $h_0$ is $M^{\prime}$-alternating and $M$-alternating.
By Lemma \ref{compatible}, $h_0\in \mathcal{A}^{\prime}$ and $h_0\in \mathcal{A}$.
See Fig. \ref{unique},
$H-e_2-e_4-e_6$ consists of three disjoint sub-catacondensed
hexagonal systems:
$H_1$, $H_2$ and $B^\prime$ (the former two may be single edges).
Note that there is a possible non-hexagon cycle $Q$ in $\mathcal{A}$ which containing  $h_0$.

If such $Q$ exists, then $Q$ must pass through $g_1$ and $f_1$ since $Q$ and $h_0$ are compatible.
So $\mathcal{A}\setminus\{Q\}$ is a compatible $M^{\prime}$-alternating set of $H^{\prime}$,
we have $|\mathcal{A}\setminus\{Q\}|\leq |\mathcal{A}^{\prime}|=i-1$.
If $|\mathcal{A}^{\prime}|=|\mathcal{A}\setminus\{Q\}|$,
then  $af(H,M)=|\mathcal{A}|=i\in$Spec$_{af}(H)$.
If  $|\mathcal{A}^{\prime}|>|\mathcal{A}\setminus\{Q\}|$, then $|\mathcal{A}^{\prime}|\geq|\mathcal{A}|$.
On the other hand, $\mathcal{A}^{\prime}$ is also a compatible $M$-alternating set, so  $|\mathcal{A}^{\prime}|\leq|\mathcal{A}|$.
We have $|\mathcal{A}|=|\mathcal{A}^{\prime}|=i-1$.
Note that $Q$ does not pass through $e_3$ and $e_5$.
Let $M_j=M\triangle h_0\triangle h_1\triangle\ldots\triangle h_j$
$(j=0,1,\ldots,r)$,
$Q_1=(E(Q)\cap E(H_1))\cup\{e_5\}$ and $Q_2=(E(Q)\cap E(H_2))\cup\{e_3\}$.
Then $Q_1$ and $Q_2$ both are $M_0$-alternating cycles.
We can see that $(\mathcal{A}\setminus\{Q\})\cup\{Q_1,Q_2,h_1\}$ is a compatible $M_0$-alternating set with cardinality $i+1$,
so $c^{\prime}(M_0)\geq i+1$.
On the other hand, $c^{\prime}(M_0)$ is at most $|\mathcal{A}|+2=i+1$.
So $c^{\prime}(M_0)=i+1$.
Note that $h_r$ is the unique $M_r$-alternating hexagon contained in $B$,
by Theorem \ref{compatibleset} and Lemma \ref{compatible}, we have $af(H,M_r)=c^{\prime}(M_r)=c^{\prime}(M_0)-1=i\in$Spec$_{af}(H)$.

If such cycle $Q$ does not exist,
then there is no cycle in $\mathcal{A}$ which passing through
the edges going out of $h_0$ since $h_0\in \mathcal{A}$.
So $\mathcal{A}$ is also a maximum compatible
$M^{\prime}$-alternating set on $H^{\prime}$,
and $|\mathcal{A}|=|\mathcal{A}^{\prime}|=i-1$.
Note that  any cycle of $\mathcal{A}\setminus\{h_0\}$ is completely contained  in $H_1$ or $H_2$.
Let $i_1$ (resp. $i_2$) be the number of cycles of $\mathcal{A}$ which are completely contained in $H_1$ (resp. $H_2$).
Then $|\mathcal{A}|=i_1+i_2+1=i-1$.
Since $M$ and $M_0$ only differ on $h_0$ and $e_5\in M_0$ (resp. $e_3\in M_0$),
the size of maximum compatible
$M_0$-alternating set on $H_1$ (resp. $H_2$)
is $i_1$ or $i_1+1$ (resp. $i_2$ or $i_2+1$).
Let $\mathcal{A}_0$ be a maximum compatible non-crossing $M_0$-alternating set of $H$ with minimal $h$-index.
Note that $h_0$ and $h_1$ both are $M_0$-alternating.
By Lemma \ref{compatible}, $h_0\in \mathcal{A}_0$ and $h_1\in \mathcal{A}_0$.
It implies that cycles in $\mathcal{A}_0\setminus\{h_0,h_1\}$
are completely contained in $H_1$ or $H_2$.
Hence $i=i_1+i_2+2\leq |\mathcal{A}_0|\leq i_1+i_2+4=i+2$.
If $|\mathcal{A}_0|=i$, then $af(H,M_0)$$=i\in$Spec$_{af}(H)$.
If $|\mathcal{A}_0|$$=i+1$,
then $(\mathcal{A}_0\setminus\{h_0,h_1\})\cup\{h_r\}$ is a maximum compatible $M_r$-alternating set with size $i$ since $M_0$ and $M_r$ only differ on $B^\prime$.
So $af(H,M_r)=i\in$Spec$_{af}(H)$.

If $|\mathcal{A}_0|=i+2$,
then $H_1$ (resp. $H_2$) contains exactly $i_1+1$ (resp. $i_2+1$) cycles of $\mathcal{A}_0$.
Let $M_0^{\prime}$ (resp. $M_0^{\prime\prime}$) be the restriction of $M_0$ to $H_1$ (resp. $H_2$).
Then $af(H_1,M_0^{\prime})=i_1+1$ (resp. $af(H_2,M_0^{\prime\prime})=i_2+1$).
Hence $af(H_1)\leq i_1+1$ (resp. $af(H_2)\leq i_2+1$).
If $af(H_1)<i_1+1$ (resp. $af(H_2)<i_2+1$),  then $i_1\in$Spec$_{af}(H_1)$ (resp. $i_2\in$Spec$_{af}(H_2)$) by the induction hypothesis.
So there is a perfect matching
$F_1$ (resp. $F_2$) of $H_1$ (resp. $H_2$) such that $c^{\prime}(F_1)=i_1$ (resp. $c^{\prime}(F_2)=i_2$).
We can see that $M_1^{*}$$=F_1$$\cup (M_r\cap (E(H_2)\cup E(B^{\prime})))$ (resp. $M_2^{*}$$=F_2$$\cup$$(M_r$$\cap$$(E(H_1)$$\cup$$E(B^{\prime})))$)
is a perfect matching of $H$ such that $af(H,M_1^{*})=c^{\prime}(M_1^{*})=i_1+i_2+1+1=i\in$Spec$_{af}(H)$
(resp. $af(H,M_2^{*})$$=$$c^{\prime}(M_2^{*})$$=i_2+i_1+1+1=$$i\in$Spec$_{af}(H)$).

Now suppose $af(H_1)=i_1+1$ and $af(H_2)=i_2+1$,
$F^{\prime}$ is a perfect matching of $H^{\prime}$ with $af(H^{\prime},F^{\prime})=a^{\prime}$.
Note that $F=F^{\prime}\cup \{f_1,f_2,\ldots,f_r,g_1,g_2,\ldots,g_r\}$ is a perfect matching of $H$.
By Lemma \ref{compatible},
either $e_1\in F^{\prime}$ or $e_4\in F^{\prime}$.
We assert that $e_4\in F^{\prime}$.
Otherwise $e_1\in F^{\prime}$,
then the restrictions $F^{\prime}_1$ and $F^{\prime}_2$ of $F^{\prime}$ to $H_1$ and $H_2$ are perfect matchings of $H_1$ and $H_2$ respectively.
So $a^{\prime}\geq c^\prime(F^{\prime}_1)+c^\prime(F^{\prime}_2)\geq af(H_1)+af(H_2)=i_1+1+i_2+1=i\in[a^{\prime}+1,Fries(H^{\prime})]$,
a contradiction.
Since $e_4\in F^{\prime}$, $\{e_2,e_6\}\subseteq F^{\prime}$.
So $h_0$ is $F^{\prime}$-alternating,
and $F^{\prime}\triangle h_0$ is a perfect matching of $H^{\prime}$.
Since the restrictions of $F^{\prime} \triangle h_0$ to $H_1$ and $H_2$ are perfect matchings of $H_1$ and $H_2$ respectively,
$c^{\prime}(F^{\prime} \triangle h_0)\geq$$ af(H_1)+af(H_2)+1=$
$i+1\geq a^\prime +2$.
On the other hand,
by Lemma \ref{compatible},
we have that $c^{\prime}(F^{\prime} \triangle h_0)\leq a^{\prime}+2$.
So $c^{\prime}(F^{\prime} \triangle h_0)=a^{\prime}+2$
and $a^{\prime}+1=i$.
Let $\mathcal{A}^{*}$ be a maximum non-crossing compatible $F^{\prime} \triangle h_0$-alternating set of $H^\prime$ with minimal $h$-index.
Then $|\mathcal{A}^{*}|$$=c^{\prime}$$(F^{\prime}$$ \triangle$$ h_0$$)=$$a^{\prime}+2$.
Note that $c^{\prime}$$(F^{\prime})$$=$$a^{\prime}$, but $c^{\prime}$$(F^{\prime}$$\triangle$$ h_0$$)=$$a^{\prime}+2$,
it implies that there must be an $F^{\prime} \triangle h_0$-alternating cycle $C_1$
(resp. $C_2$) of $\mathcal{A}^{*}$ passing through $e_5$ (resp. $e_3$) and contained  in $H_1$ (resp. $H_2$).
We can see that $D=(C_1-e_5)\cup(C_2-e_3)\cup\{e_2,e_4,e_6,g_1,f_1,g_1^{\prime},f_1^{\prime},d_1\}$ is an $F$-alternating cycle containing $h_0$,
and $D$ is compatible with each cycle of $\mathcal{A}^{*}\setminus\{C_1,C_2\}$.
So $\{D\}\cup (\mathcal{A}^{*}\setminus\{C_1,C_2\})$ is a compatible $F$-alternating set
with size $i$.
By theorem \ref{compatibleset},
$af(H,F)\geq i$.
On the other hand,
by Theorem \ref{compatibleset} and Lemma \ref{compatible},
we have $i=a^\prime+1=c^\prime(F^\prime)+1\geq c^\prime(F)=af(H,F)$.
Therefore,
$af(H,F)=i\in$Spec$_{af}(H)$.

By the arbitrariness of $i$, we proved that  $[a^{\prime}+1,Fries(H^{\prime})]\subseteq$Spec$_{af}(H)$.

\noindent\textbf{Claim 2.} $a^{\prime}\leq af(H)$.

Let $af(H)=a$,
$M$ be a perfect matching of $H$ with $af(H,M)=a$.
By Lemma \ref{unique1}, just one edge of $\{e_4,e_1,d_1,d_2,\ldots,d_r\}$ belongs to $M$.
If $e_4\in M$ or $e_1\in M$,
then the restriction $M^{\prime}$ of $M$ to $H^{\prime}$ is a perfect matching of $H^{\prime}$.
So $a^\prime\leq c^\prime(M^\prime)\leq c^\prime(M)=a$.
If $d_i\in M$ $(1\leq i\leq r)$, then $h_i$ is $M$-alternating.
Let $M_i$$=$$M$$\triangle$$h_i$$\triangle$$h_{i-1}$$\triangle$$\cdots$$\triangle$
$h_1$.
Then $M_i$ is a perfect matching of $H$ and $e_1\in M_i$.
Note that $M$ and $M_i$  only differ  on $B^\prime$.
By Lemma \ref{compatible}, $c^\prime(M_i)\leq c^\prime(M)+1=a+1$.
Since $e_1\in M_i$,
the restriction $M_i^{\prime}$ of $M_i$ to $H^{\prime}$ is a perfect matching of $H^{\prime}$.
So $a^\prime\leq c^\prime(M_i^{\prime})$.
Let $\mathcal{A}_i^{\prime}$  be a maximum non-crossing compatible $M_i^{\prime}$-alternating set
with minimal $h$-indices in $H^{\prime}$.
Then $\mathcal{A}_i^{\prime}\cup \{h_1\}$ is a compatible $M_i$-alternating set of $H$ since $h_1$ is $M_i$-alternating.
Hence $|\mathcal{A}_i^{\prime}\cup \{h_1\}|\leq c^\prime(M_i)$.
Further, $a^{\prime}+1\leq c^\prime(M_i^{\prime})+1=|\mathcal{A}_i^{\prime}|+1\leq c^\prime(M_i)\leq a+1$,
i.e., $a^{\prime}\leq a$.

\noindent\textbf{Claim 3.} $Fries(H)\leq Fries(H^{\prime})+2$.

Let $M$ be a perfect matching of $H$
and $h^\prime(M)$ denote the the number of $M$-alternating hexagons of $H$.
Suppose $h^\prime(M)=Fries(H)$.
By Lemma \ref{unique1},
only one of $\{e_4,e_1,d_1,d_2,\ldots,d_r\}$ belongs to $M$.
If $e_1\in M$ or $e_4\in M$,
then the restriction $M^\prime$ of $M$ to $H^\prime$
is a perfect matching of $H^\prime$.
Note that $B^\prime$ contains at most one $M$-alternating hexagon.
So $h^\prime(M)\leq h^\prime(M^\prime)+1\leq Fries(H^\prime)+1$.
If $d_i\in M$ ($1\leq i\leq r$),
then $h_i$ is $M$-alternating.
Let $M_i$ $=$$M$$\triangle$$h_i$$\triangle$$h_{i-1}$$\triangle$$\cdots$$\triangle$
$h_1$.
Then $M_i$ is a perfect matching of $H$ such that the restriction $M_i^\prime$ of $M_i$ to  $H^\prime$ is a perfect matching of $H^\prime$.
Since $h_1$ is the unique $M_i$-alternating hexagon contained in $B^\prime$,
$h^\prime(M_i)=h^\prime(M_i^\prime)+1$.
Note that $h^\prime(M_i)=h^\prime(M)$ or $h^\prime(M_i)=h^\prime(M)-1$
since $B$ is a linear chain.
If $h^\prime(M_i)=h^\prime(M)$,
then $h^\prime(M)=h^\prime(M_i^\prime)+1\leq Fries(H^\prime)+1$.
If $h^\prime(M_i)=h^\prime(M)-1$,
then $h^\prime(M)=h^\prime(M_i^\prime)+2\leq Fries(H^\prime)+2$.

\noindent\textbf{Claim 4.} If $Fries(H)=Fries(H^{\prime})+2$, then $Fries(H^{\prime})+1\in$Spec$_{af}(H)$.

The proof of Claim 3 implies that
if $Fries(H)=Fries(H^{\prime})+2$,
then there is a perfect matching $M$ of $H$
with $h^\prime(M)=Fries(H)$ such that just two adjacent $M$-alternating
hexagons $h_{i}$ and $h_{i+1}$ ($1\leq i<r$) contained in $B^\prime$.
Let $M^{*}=M\triangle h_{i+1}\triangle h_{i+2}\triangle\cdots\triangle h_r$.
Then $M^{*}$ is a perfect matching of $H$ and
$h_r$ is the unique $M^*$-alternating hexagon contained in $B^\prime$.
By Lemma \ref{compatible},
$c^{\prime}(M^{*})=h^{\prime}(M)-1=Fries(H^{\prime})+1$.
By Theorem \ref{compatibleset},
$af(H,M^{*})=Fries(H^{\prime})+1\in$Spec$_{af}(H)$.

We can see that Claims 1, 2, 3 and 4 imply that there is no gap in the anti-forcing spectrum of $H$.
\end{proof}

According to Theorems \ref{maxantiforcing} and \ref{continuous}, the following corollary is immediate.
\begin{corollary}{\bf .}
Let $H$ be a cata-condensed hexagonal system. Then Spec$_{af}(H)=[af(H),Fries(H)]$.
\end{corollary}

\end{document}